\documentclass[11pt,letterpaper]{amsart}
\usepackage{titletoc}
\usepackage{blindtext}
\usepackage{amsfonts,amssymb,amsmath,amsthm}
\usepackage{enumerate}

\newtheorem{innercustomgeneric}{\customgenericname}
\providecommand{\customgenericname}{}
\newcommand{\newcustomtheorem}[2]{%
  \newenvironment{#1}[1]
  {%
   \renewcommand\customgenericname{#2}%
   \renewcommand\theinnercustomgeneric{##1}%
   \innercustomgeneric
  }
  {\endinnercustomgeneric}
}

\newcustomtheorem{customthm}{Theorem}
\newcustomtheorem{customlemma}{Lemma}

\usepackage{graphicx, color}

\makeatletter
\@namedef{subjclassname@2020}{\textup{2020} Mathematics Subject Classification}
\makeatother

\makeatletter

\newtheorem{lem}{Lemma}

\newtheorem{definition}{Definition}
\newtheorem*{problem}{Conjecture}

\numberwithin{equation}{section}
\begin{document}
\title[Topological shadowing for linear operators]{Topological shadowing for linear operators}

\author{Y. Yang}

\address{School of Mathematical Sciences,
Liaoning University, Shenyang 1000191, Peoples Republic of China.}\email{yynmath@163.com}

\author[C.A. Morales]{C.A. Morales}
\address{Hangzhou International Innovation Institute of Beihang University,  
Hangzhou 311115, China.}
\email{morales@impa.br}


\begin{abstract}
We prove a linear operator of a finite-dimensional Banach space has the topological shadowing property, as introduced in \cite{lny}, if and only if its spectrum lies either within the open unit complex disk or outside the closure of that disk.
Moreover, the spectrum of every uniformly expansive linear operator with the topological shadowing property lies either within the open unit complex disk or outside the closure of that disk.
Lastly, we prove that a normal operator with the topological shadowing property on a Hilbert space does not have any nonzero nonwandering points.
\end{abstract}

\maketitle



\section{Introduction}

\noindent
Let \( X \) be a metric space. Denote by \( C^+(X) \) the set of all continuous functions \( \delta: X \to (0, \infty) \). Let \( f: X \to X \) be a bijective map.
For a given \( \delta \in C^+(X) \), we say that a sequence \( (x_n)_{n \in \mathbb{Z}} \) is a {\em \(\delta\)-pseudo orbit} if
$$
d(f(x_n), x_{n+1}) < \delta(f(x_n))\quad\quad\forall n \in \mathbb{Z}.
$$
Furthermore, \( (x_n)_{n \in \mathbb{Z}} \) can be {\em \(\delta\)-shadowed} if there exists \( x \in X \) such that
$$
d(f^n(x), x_n) < \delta(f^n(x))
\quad\quad\forall n \in \mathbb{Z}.
$$

\begin{definition}
We say that \( f \) has the {\em topological shadowing property} if for every \( \epsilon \in C^+(X) \), there exists \( \delta \in C^+(X) \) such that every \( \delta \)-pseudo orbit can be \( \epsilon \)-shadowed.
\end{definition}

This definition was introduced in \cite{lny} under the name "shadowing property." Here, we adopt the terminology from \cite{dlrw}, which is also used in \cite{c}, to avoid confusion with the classical {\em shadowing property} \cite{b}, where the \(\epsilon\) and \(\delta\) mentioned above are required to be constant functions.

Now, assume that \( X \) is a (complex) Banach space with positive dimension (or nontrivial for short). By a {\em linear operator}, we mean an invertible bounded linear operator \( L: X \to X \) with a bounded inverse \( L^{-1}: X \to X \). Let \( GL(X) \) denote the set of all such operators. The {\em spectrum} of \( L \) is the set $\sigma(L)$ of complex numbers \( \lambda \) for which the operator \( L - \lambda I \) is not invertible. We denote by \( \mathbb{D} \) the open unit disk in \( \mathbb{C} \) and by \( \overline{\mathbb{D}} \) the closure of \( \mathbb{D} \).

Some recent works have studied the shadowing property for linear operators \cite{bm}, \cite{bcdmp}. Regarding topological shadowing, we believe this property is very restrictive for such operators. This idea is encapsulated in the following conjecture:

\begin{problem}
A linear operator on a Banach space has the topological shadowing property if and only if its spectrum lies in either $\mathbb{D}$ or $\mathbb{C} \setminus \overline{\mathbb{D}}$.
\end{problem}

Our first result proves this conjecture in the finite-dimensional case.

\begin{customthm}{A}
\label{c0}
A linear operator of a finite-dimensional Banach space has the topological shadowing property if and only if its spectrum lies in either $\mathbb{D}$ or $\mathbb{C} \setminus \overline{\mathbb{D}}$.
\end{customthm}

Now, recall that a linear operator $L:X\to X$ is {\em uniformly expansive} \cite{h} if there $m\in\mathbb{N}$ such that
$\|L^m(x)\|\geq 2$ or $\|L^{-n}(x)\|\geq 2$ for all unitary $x\in X$.

\begin{customthm}{B}
\label{c-1}
The spectrum of every uniformly expansive linear operator with the topological shadowing property lies in either $\mathbb{D}$ or $\mathbb{C} \setminus \overline{\mathbb{D}}$.
\end{customthm}

Recall the {\em nonwandering set} $\Omega(L)$ of $L\in GL(X)$ consisting of those $x\in X$ such that
for every neighborhood $U$ of $x$ there is $n\in\mathbb{N}$ such that
$L^n(U)\cap U\neq\emptyset$.
A linear operator of a Hilbert space is {\em normal} if it commutes with its adjoint operator.

Mazur's Theorem \cite{ma} asserts that a normal operator of a Hilbert space has the shadowing property if and only if it is hyperbolic.
Then, a normal operator $L$ with the shadowing property of a Hilbert space has no nonzero nonwandering points, i.e. $\Omega(L)=\{0\}$.
Below we obtain the same conclusion but for normal operators with the topological shadowing property instead.

\begin{customthm}{C}
\label{c-2}
If $L$ is a normal operator with the topological shadowing property of a Hilbert space, then $\Omega(L)=\{0\}$.
\end{customthm}

To prove these results we obtain that if a linear operator of a (possibly infinite dimensional) Banach space is hyperbolic and has the topological shadowing property, then its spectrum lies in either $\mathbb{D}$ or $\mathbb{C}\setminus \overline{\mathbb{D}}$ (Lemma \ref{sim}).
It would be also interesting to prove that
every linear operator with the topological shadowing property of a Banach space has the shadowing property (this would be consequence of the above conjecture).
The rest of the paper is divided as follows. In Section \ref{sec2}, we prove some preliminary lemmas. In Section \ref{sec3}, we prove the theorems.

\section{Preliminary lemmas}
\label{sec2}

\noindent
The following lemma implicitly used elsewhere (e.g. Remark 2.6 in \cite{c}) is a direct consequence of Lemma 2.8 in \cite{lny}.
We state and prove it here for completeness.

\begin{lem}
\label{landru}
For every metric space $X$ and $\alpha\in C^+(X)$ there exists
$\beta\in C^+(X)$ such that
if $x,y\in X$ and $d(x,y)< \beta(x)$, then $d(x,y)< \alpha(y)$.
\end{lem}

\begin{proof}
Lemma 2.8 in \cite{lny} provides $\beta\in C^+(X)$ such that
$\beta(x)<\inf\{\alpha(y):y\in B(x,\beta(x))\}$ for all $x\in X$.
Then, if $d(x,y)<\beta(x)$, $y\in B(x,\beta(x))$ so $\beta(x)<\alpha(y)$ thus
$d(x,y)<\alpha(y)$ as desired.
\end{proof}

Next, we state two auxiliary variations of the topological shadowing property.
For the first one, given a bijective map of a metric space $g:Y\to Y$ and $\delta\in C^+(Y)$, 
we say that a sequence a $(x_n)_{n\geq0}$ is a {\em positive $\delta$-pseudo orbit} if
$$
d(g(x_n),x_{n+1})< \delta(g(x_n)),\quad\quad\forall n\geq0.
$$
We say that $(x_n)_{n\geq0}$ can be {\em $\delta$-shadowed} if there is $x\in Y$ such that
$$
d(g^n(x),x_n)<\delta(g^n(x)),\quad\quad\forall n\geq0.
$$
We shall use the following definition (See Definition 3.2 in \cite{c}):

\begin{definition}
A homeomorphism of a metric space $g:Y\to Y$ has the {\em topological positive shadowing property} if for every $\epsilon\in C^+(Y)$ there is $\delta\in C^+(Y)$ such that every positive $\delta$-pseudo orbit of $g$ can be $\epsilon$-shadowed.
\end{definition}

The following lemma can be proved as in Lemma 3.2 in \cite{c}:

\begin{lem}
\label{preparada}
A homeomorphism of a locally compact metric space has the topological shadowing property if and only if it has the topological positive shadowing property.
\end{lem}

To motivate the second variation we remind that a bijective map of a metric space has the {\em bounded shadowing property} if for every $\epsilon>0$ there is $\delta>0$ such that every bounded $\delta$-pseudo orbit can be $\epsilon$-shadowed \cite{lm}.
By replacing $\epsilon$ and $\delta$ in this definition by positive functions we obtain the following concept.

\begin{definition}
A bijective map of a metric space $g:Y\to Y$ has the {\em topological bounded shadowing property}
if for every $\epsilon\in C^+(Y)$ there is $\delta\in C^+(Y)$ such that
every {\em bounded} $\delta$-pseudo orbit of $g$ can be $\epsilon$-shadowed.
\end{definition}

It is clear that the the topological shadowing property implies the topological bounded shadowing property. We don't know if the converse holds.

On the other hand, given again a bijective map of a metric space $g:Y\to Y$ we define
$$
E^c_g=\{y\in Y:(g^n(y))_{n\in\mathbb{Z}}\mbox{ is bounded}\}.
$$
(We will drop the subindex $g$ above to write $E^c$ instead).
Clearly $g(E^c)=E^c$ so the restriction $g|_{E^c}:E^c\to E^c$ is well-defined.
We would like to prove that if $g$ has the topological bounded shadowing property, then so does $g|_{E^c}$. This is motivated by Lemma 5 in \cite{lm}.

We obtain the follows partial answer that suffices for our purpose:

\begin{lem}
\label{viva}
Let $g:Y\to Y$ be a homeomorphism with the topological bounded shadowing property of a metric space.
If $E^c$ is closed, then $g|_{E^c}:E^c\to E^c$ also has the topological bounded shadowing property.
\end{lem}

\begin{proof}
We first make some observations about the extension of continuous functions on metric spaces.
Tietze's extension theorem says that for every closed subset $E\subset Y$ and every continuous map $h:E\to \mathbb{R}$ there is $H:Y\to \mathbb{R}$ continuous such that $H|_E=h$. Moreover, if $h$ is bounded, then $H$ can be chosen to be bounded too \cite{ke}.

Now, we assert that if $h$ is bounded and {\em positive}, then $H$ can be chosen to be bounded and positive too.
In fact, if $h$ is continuous, positive and bounded, then $\log h:E\to \mathbb{R}$ is well-defined, continuous and bounded.
Tietze's provides $F:Y\to\mathbb{R}$ continuous and bounded such that
$F|_E=\log h$. So, $H=e^F:X\to \mathbb{R}$ is continuous, bounded, positive
and $H|_E=h$. This proves the assertion.

Next, we fix $\epsilon^c\in C^+(E^c)$. By replacing it by $\min\{1,\epsilon^c\}$ if necessary, we can assume that $\epsilon^c$ is bounded. Then, the assertion above provides
$\epsilon\in C^+(Y)$ bounded such that $\epsilon|_{E^c}=\epsilon^c$. For this $\epsilon$ we choose
$\delta\in C^+(Y)$ from the topological shadowing of $g$, and define
$\delta^c=\delta|_{E^c}\in C^+(E^c)$.

Let $(y_n^c)_{n\in\mathbb{Z}}$ be a bounded $\delta^c$-pseudo orbit of $g|_{E^c}$. In particular, $(y^c_n)_{n\in\mathbb{Z}}$ is a $\delta$-pseudo orbit of $g$ so the topological bounded shadowing of $g$ provides
$y\in Y$ such that
$$
d(g^n(y),y^c_n)\leq \epsilon(g^n(y)),\quad\quad\forall n\in\mathbb{Z}.
$$
Since $(y^c_n)_{n\in\mathbb{Z}}$ is bounded, there is $\Delta>0$ such that
$$
d(y_1^c,y^c_n)\leq \Delta,\quad\quad\forall n\in\mathbb{Z}.
$$
Then,
$$
d(g^n(y),y^c_1)\leq d(g^n(y),x^c_n)+d(x^c_n,x^c_1)\leq \epsilon(g^n(y))+\Delta,\quad\quad\forall n\in\mathbb{Z}.
$$
Since $\epsilon$ is bounded, the above shows that $(g^n(y))_{n\in\mathbb{Z}}$ is bounded so $y\in E^c$ thus $g|_{E^c}$ has the topological bounded shadowing property.
This completes the proof.
\end{proof}

It is well-known that the topological shadowing property is an invariant under topological conjugacy \cite{dlrw}, \cite{lny}.
More precisely, if $Y$ and $Z$ are metric spaces, $H:Y\to Z$ is a homeomorphism and if $P:Z\to Z$ is a homeomorphism with the topological shadowing property, then so does $H\circ P\circ H^{-1}:Y\to Y$.
We need a similar property for the topological bounded shadowing property.
Recall that a map between metric spaces $H:Y\to Z$ is {\em Lipschitz} if there is
a positive real number $K$ such that
$d(H(y),H(y'))\leq Kd(y,y')$ for all $y,y'\in Y$.
We say that $H$ is a {\em Lipeomorphism} if it is Lipschitz with Lipschitz inverse.
By using Lemma 2.7 in \cite{lny} and the proof of Lemma 6 in \cite{lm} we obtain the following lemma.

\begin{lem}
\label{horizon}
Let $H:Z\to Y$ be a Lipeomorphism of metric spaces.
If $P:Z\to Z$ is a homeomorphism with the topological bounded shadowing property, then so does $H\circ P\circ H^{-1}:Y\to Y$.
\end{lem}

To state the next lemma, given a Banach space $X$ we say that $L\in GL(X)$ is a {\em linear isometry}
if $\|L(x)\|=\|x\|$ for all $x\in X$.

It is well-known that a linear isometry has not the bounded shadowing property \cite{lm}.
A slight modification of the proof in \cite{lm} permits to extend
this conclusion to the topological bounded shadowing property. More precisely, we have the following lemma.

\begin{lem}
\label{l1}
A linear isometry of a nontrivial Banach space does not have the topological bounded shadowing property.
\end{lem}

\begin{proof}

Suppose by contradiction that there is a linear isometry of a nontrivial Banach space $L:X\to X$ with the topological bounded shadowing property.
Take $\delta\in C^+(X)$ from this property for $\epsilon=1$,
and let $x\in X$ be an arbitrary point.
Choose a sequence
$0=q_0, q_1,\cdots, q_r=x$ such that
$\|q_{i+1}-q_i\|\leq\delta(L^{i+1}(q_i))$ for $0\leq i\leq r-1$.
Define $(p_i)_{i\in\mathbb{Z}}$ by
$$
p_i = \left\{ \begin{array}{rcl}
L^i(q_0),& \mbox{if} & i<0\\
& & \\
L^i(q_i), & \mbox{if} & 0\leq i\leq r\\
& & \\
L^i(q_r), & \mbox{if} & r<i.
\end{array} \right.
$$
Since
$$
\sup_{i<0}\|L^i(q_0)\|=\|q_0\|<\infty\quad\mbox{ and }\quad
\sup_{i>r}\|L^i(q_r)\|=\|q_r\|<\infty,
$$
one has that the sequence $(p_i)_{i\in\mathbb{Z}}$ is bounded.
Moreover,
$$
\|L(p_i)-p_{i+1}\|=\|L^{i+1}(q_i)-L^{i+1}(q_{i+1})\|=\|q_{i+1}-q_i\|\leq\delta(L^{i+1}(q_i))=\delta(L(p_i))
$$
for all $0\leq i\leq r-1$.

So, $(p_i)_{i\in\mathbb{Z}}$ is a bounded
$\delta$-pseudo orbit of $L$.
Then, the topological bounded shadowing provides
$z\in X$ such that
$$
\|L^i(z)-p_i\|\leq 1,\quad\quad\forall i\in\mathbb{Z}.
$$
Since $L$ is a linear isometry,
$$
\|z-q_i\|=\|z-L^{-i}(p_i)\|=\|L^i(z)-p_i\|\leq1,\quad\quad\forall 0\leq i\leq r.
$$
In particular,
$\|z\|=\|z-0\|=\|z-q_0\|\leq1$ and $\|z-x\|=\|z-q_r\|\leq1$ so
$$
\|x\|\leq \|z\|+\|z-x\|\leq 2,\quad\quad\forall x\in X.
$$
So, $X=\{0\}$ thus $dim(X)=0$ a contradiction. This completes the proof.
\end{proof}

A linear operator of a Banach space $L:X\to X$ is
{\em expansive} if for every unitary $x\in X$ there is $n\in\mathbb{Z}$ such that $\|L^n(x)\|\geq2$.

\begin{lem}
\label{baja}
Let $L:X\to X$ be a linear operator with the topological shadowing property of a Banach space. If $E^c$ is closed, then $L$ is expansive.
\end{lem}

\begin{proof}
The argument is similar to that in Proposition 10 in \cite{lm} we add some details for completeness.

By Proposition 19 of \cite{bcdmp} we only need to prove that $E^c$ is trivial (i.e. $E^c=\{0\}$).
Suppose by contradiction that $E^c$ is not trivial.
Since $E^c$ is a subspace, $E^c$ is a closed subspace and also it is invariant i.e.
$L(E^c)=E^c$.
Then, $P\in GL(Z)$ where $Z$ is $E^c$ equipped with the norm $\|\cdot\|$ induced by $X$
and $P=L|_{E^c}$.
We also have $\sup_{n\in \mathbb{Z}}\|P^n(z)\|<\infty$ for all $z\in Z$
hence $M=\sup_{n\in\mathbb{Z}}\|P^n\|<\infty$ by the Banach-Steinhouse Theorem.
This allows us to define the norm
$\|z\|'=\sup_{n\in\mathbb{Z}}\|P^n(z)\|$ for $z\in Z$ which clearly satisfies
$\|z\|\leq \|z\|'\leq M\|z\|$ for $z\in Z$.
We set $Y=Z$ endowed with the norm $\|\cdot\|'$ thus $Y$ is a Banach space.
We define $H:Z\to Y$ as the identity hence $H$ is a Lipeomorphism.
Since $P$ has the topological bounded shadowing property, $H\circ P\circ H^{-1}$ also does by Lemma \ref{horizon}.
However,
$\|H\circ P\circ H^{-1}\|'=\|L(z)\|'=\sup_{n\in\mathbb{Z}}\|L^{n+1}(z)\|=\|z'\|$ for all $z\in Y$ proving that $H\circ P\circ H^{-1}$ is a linear isometry of $Y$.
This contradicts Lemma \ref{l1} since $E^c$ is nontrivial, so $E^c$ must be trivial completing the proof.
\end{proof}

Recall that a linear operator of a Banach space is {\em hyperbolic} if its spectrum does not intersect the unit circle.
The proof of the following lemma is based on the proof of Theorem 3.3 in \cite{c}.

\begin{lem}
\label{sim}
If a hyperbolic linear operator of a Banach space
has the topological shadowing property, then its spectrum lies in either $\mathbb{D}$ or $\mathbb{C} \setminus \overline{\mathbb{D}}$.
\end{lem}

\begin{proof}
Suppose by contradiction that there is a hyperbolic linear operator of a Banach space $L:X\to X$ which has the topological shadowing property
but its spectrum does not lie either in $\mathbb{D}$ or in $\mathbb{C}\setminus \overline{\mathbb{D}}$.
Then, from the existence of adapted norms (e.g. Lemma 1 in \cite{eh}) we can assume that there are a direct sum by closed nonzero subspaces $X=N\oplus M$ and a constant $0<\lambda<1$ such that
$L(N)=N$, $L(M)=M$,
$\|L(x)\|\geq \lambda^{-1} \|x\|$ for all $x\in N$ and $\|L(x)\|\leq\lambda\|x\|$ for all $x\in M$.
Furthermore,
$$
\|x+y\|=\max\{\|x\|,\|y\|\},\quad\quad\forall x+y\in N\oplus M.
$$
Define $\epsilon\in C^+(X)$ by
$$
\epsilon(z)=\frac{1}{2^{\|z\|}},\quad\quad\forall z\in X.
$$

For this $\epsilon$ we choose $\epsilon'\in C^+(X)$ such that Lemma \ref{landru} holds for $\alpha=\epsilon$ and $\beta=\epsilon'$.
For this $\epsilon'$ we let $\delta\in C^+(X)$ be given by the topological shadowing of $L$.
For this $\delta$ we choose $\delta'\in C^+(X)$ such that Lemma \ref{landru}
holds for $\alpha=\delta$ and $\beta=\delta'$.

Since both $M$ and $N$ are nonzero, and $\delta'(0)>0$, we can choose
$x_0^N\in N\setminus\{0\}$ and $y_0^M\in M\setminus \{0\}$ such that
$\|y^M_0\|<\delta'(x_0^N)$.
Define the sequence $(z_n)_{n\in\mathbb{Z}}$ by
$$
z_n= \left\{
\begin{array}{rcl}
L^n(x_0^N+y_0^M),  & \mbox{if} & n<0\\
& & \\
L^n(x_0^N),& \mbox{if} & n\geq0.
\end{array}
\right.
$$
It follows that
$\|L(z_n)-z_{n+1}\|<\delta'(z_{n+1})$ for all $n\in\mathbb{Z}$, and so,
$\|L(z_n)-z_{n+1}\|<\delta(L(z_n))$ for all $n\in\mathbb{Z}$ (by Lemma \ref{landru}).
We conclude that $(z_n)_{n\in\mathbb{Z}}$ is a $\delta$-pseudo orbit. Then,
there is $\bar{x}+\bar{y}\in N\oplus M$ such that
$\|L^n(\bar{x}+\bar{y})-z_n\|<\epsilon'(L^n(\bar{x}+\bar{y}))$ for all
$n\in\mathbb{Z}$.
So, Lemma \ref{landru} and the choice of $\epsilon'$ imply
\begin{equation}
\label{pupa}
\|L^n(\bar{x}+\bar{y})-z_n\|<\epsilon(z_n),\quad\quad\forall
n\in\mathbb{Z}.
\end{equation}

We claim that $\bar{x}=x_0^N$.
Suppose not. Since $z_n=L^n(x_0^N)$ for $n>0$, one has
\begin{equation}
\label{janga}
\|L^n(\bar{x}+\bar{y})-z_n\|=\max\{\|L^n(x_0^N-\bar{x})\|,L^n(\bar{y})\|\},\quad\quad\forall n>0.
\end{equation}
Then, since $\bar{x}\neq x^N_0$, $\|L^n(x_0^N-\bar{x})\|\geq \lambda^{-n}\|x^N_0-\bar{x}\|$ and $\|L^n(\bar{y})\|\leq\lambda^n\|\bar{y}\|$ for all $n\geq1$, there would exist $n_0\in\mathbb{N}$ such that
$$
\| L^n(\bar{x}+\bar{y})-z_n\|=\|L^n(x_0^N-\bar{x})\|\geq \lambda^{-n}\|x_0^N-\bar{x}\|,\quad\quad\forall n\geq n_0.
$$
But
$$
\epsilon(z_n)=\epsilon(L^n(x^N_0))=\frac{1}{2^{\|L^n(x^N_0)\|}},
\quad\quad\forall n>0,
$$
so replacing in \eqref{pupa} we get
$$
\|x_0^N-\bar{x}\|<\frac{\lambda^n}{2^{\lambda^{-n}\|x^N_0\|}},\quad\quad\forall n>n_0.
$$
Letting $n\to\infty$ we get $\bar{x}=x^N_0$ which is absurd.
This contradiction proves the claim.

Now, the claim and \eqref{janga} yield
$$
\|L^n(\bar{x}+\bar{y})-z_n\|=\|L^n(\bar{y})\|,
\quad\quad\forall n>0.
$$
It follows that
$$
\|L^n(\bar{y})\|=\|L^n(\bar{x}+\bar{y})-z_n\|<\epsilon(z_n)=\epsilon(L^n(x^N_0))=\frac{1}{2^{\|L^n(x_0^N)\|}}\leq \frac{1}{2^{\lambda^{-n}\|x_0^N\|}},\,\forall n>0,
$$
so
$$
\|\bar{y}\|\leq \|L^{-n}\|\|L^n(\bar{y})\|\leq \frac{\|L^{-1}\|^n}{2^{\lambda^{-n}\|x_0^N\|}},\quad\quad\forall n>0.
$$
Letting $n\to\infty$ above yields $\bar{y}=0$. Henceforth,
$$
\bar{x}+\bar{y}=x^N_0.
$$
Finally, we observe that $L^n(x^N_0)\to 0$ as $n\to-\infty$
whereas
$$
\|z_n\|=\|L^n(x^N_0+y^M_0)\|=\max\{\|L^n(x^N_0)\|,\|L^n(y^M_0)\|\}\geq \|L^n(y^M_0)\|\to\infty
$$
as
$n\to-\infty$
so
$$
\frac{1}{2^{\|z_n\|}}\geq \|L^n(\bar{x}+\bar{y})-z_n\|=\|L^n(x^N_0)-z_n\|
\geq \|z_n\|-\|L^n(x^N_0)\|\to\infty\mbox{ as } n\to-\infty
$$
that's absurd. This completes the proof.
\end{proof}

It is known that a linear operator of a Banach space is hyperbolic if and only if it is expansive and has the shadowing property \cite{bm}.
The arguments in this reference can be used to prove the following lemma.

\begin{lem}
\label{putin}
Every uniformly expansive linear operator with the topological shadowing property of a Banach space is hyperbolic.
\end{lem}

\begin{proof}
Let $L:X\to X$ be a uniformly expansive linear operator of a Banach space.
Define
$$
E=\{x\in X:(\|L^n(x)\|)_{n\geq0}\mbox{ is bounded}\}
$$
and
$$
F=\{x\in X:(\|L^{-n}(x)\|)_{n\geq0}\mbox{ is bounded}\}.
$$
Clearly, these are subspaces of $X$, $L(E)=E$ and $L(F)=F$.
Since $L$ is uniformly expansive, by Proposition 2 in \cite{bm}, there are $c\geq1$ and $0<\beta<1$ such that
$
E=\{x\in X:\|L^n(x)\|\leq c\beta^n\|x\|,\, \forall n\geq0\}
$
and
$
F=\{x\in X:\|L^{-n}(x)\|\leq c\beta^n\|x\|,\, \forall n\geq0\}.
$
From this it follows that $E$ and $F$ are closed subspaces. Moreover,
$E\cap F=\{0\}$ and by the spectral radio theorem we also have
$r(L|_E)\leq \beta<1$ and $r(L^{-1}|_F)\leq\beta<1$
(here $r(\cdot)$ denotes the spectral radio operation).

Next prove that $X=E+F$.
Consider the constant function $\epsilon=1$ in $C^+(X)$.
For this function we let $\delta\in C^+(X)$ be given by the topological shadowing property of $L$. Choose $\bar{\delta}\in C^+(X)$ such that
$\alpha=\delta$ and $\beta=\bar{\delta}$ satisfies Lemma \ref{landru}.
Take $x\in B(0,\bar{\delta}(0))$ and define
$$
z_n= \left\{
\begin{array}{rcl}
L^n(x),  & \mbox{if} & n<0\\
0,& \mbox{if} & n\geq0.
\end{array}
\right.
$$
It follows that
$
\|L(z_{-1})-z_0\|=\|L(L^{-1}(x))\|=\|x\|<\bar{\delta}(0)=\bar{\delta}(z_0)
$
so
$
\|L(z_{-1})-z_0\|<\delta(L(z_{-1}))
$ by Lemma \ref{landru}.
Then,
$(z_n)_{n\in\mathbb{Z}}$ is a $\delta$-pseudo orbit and so
there is $z\in X$ such that
$\|L^n(z)-z_n\|\leq 1$ for every $n\in\mathbb{Z}$.
It follows that $\|L^n(x-z)\|\leq 1$ for all $n<0$ hence $x-z\in F$.
We also have $\|L^n(z)\|\leq 1$ for $n>0$ so $z\in E$.
Since $x=z+(x-z)\in E+F$, we conclude that $x\in E+F$ proving
$B(0,\bar{\delta}(0))\subset E+F$.
Since $E+F$ is a subspace, $X=E+F$ hence $X=E\oplus F$.
We have then shown that there is a direct sum by closed subspaces $X=E\oplus F$ such that $L(E)=E$, $L(F)=F$, $r(L|_E)<1$ and $r(L^{-1}|_F)<1$. It is well-known that this implies that $L$ is hyperbolic hence we are done.
\end{proof}

The converse of the above lemma is false. More precisely, there are hyperbolic linear operators without the topological shadowing property:
just consider one which is neither a uniform contraction nor a uniform contraction
and apply Lemma \ref{sim}.

\section{Proof of the theorems}
\label{sec3}

\begin{proof}[Proof of Theorem \ref{c0}]\,
Let $L:X\to X$ be a linear operator of a finite-dimensional Banach space.
First assume that $L$ has the topological shadowing property.
Then, $L$ has the topological bounded shadowing property.
Since every subspace of a finite-dimensional Banach space is closed,
$E^c$ is a closed subspace of $X$. Then, $L$ is expansive by Lemma \ref{baja}
and so hyperbolic since $dim(X)<\infty$.
Therefore, the spectrum of $L$ lies in either $\mathbb{D}$ or $\mathbb{C}\setminus \overline{\mathbb{D}}$ by Lemma \ref{sim}.

To prove the converse, we
first assume that the spectrum of $L$ lies in $\mathbb{C}\setminus \overline{\mathbb{D}}$.
Then, there are $0<\lambda<1$ and an equivalent metric still denoted by $\|\cdot\|$ such that $\|L(x)\|\geq \lambda^{-1}\|x\|$ for all $x\in X$.
From now on we proceed as in the proof of Theorem 3.5 in \cite{c}:
Let $\epsilon\in C^+(X)$. Choose $\bar{\epsilon}\in C^+(X)$ such that
$\alpha=\epsilon$ and $\beta=\bar{\epsilon}$ satisfy the conclusion of Lemma
\ref{landru}. 
Since $X$ is finite-dimensional, the closed ball $B[0,\bar{\epsilon}(0)]$ is compact so the number
$$
m=(\lambda^{-1}-1)\min\{\bar{\epsilon}(x):x\in B[0,\bar{\epsilon}(0)]\}
$$
is positive and finite. We further define
$$
K=\frac{2\|L\|}{\lambda^{-1}-1}.
$$
Choose $\delta\in C^+(X)$ satisfying the following properties:
\begin{enumerate}
\item[(i)]
$\delta<\min\{(\lambda^{-1}-1)\bar{\epsilon},m\}$ (in particular $\delta$ is bounded);
\item[(ii)]
$\delta(x)<\frac{\|x\|}K$ for all $x\notin B[0,\bar{\epsilon}(0)]$;
\item[(iii)]
$\|x\|<\|x'\|$ implies $\delta(x)>\delta(x')$.
\end{enumerate}
Let $(x_n)_{n\geq0}$ be a $\delta$-pseudo orbit.
Then, we can rewrite $x_n$ as
$x_n=L^n(x_0)+\sum_{i=0}^{n-1}L^i(r_{n-i})$, $\forall n\geq 1,$
for some sequence $(r_n)_{n\geq1}$ in $X$ satisfying
$\|r_n\|\leq \delta(L(x_{n-1}))$ for all $n\geq 1$.
Define $(\bar{x}_n)_{n\geq1}$ by
$\bar{x}_n=L^{-n}(x_n)$ namely
$$
\bar{x}_n=x_0+\sum_{i=1}^nL^{-i}(r_i),\quad\quad \forall n\geq1.
$$
Since $\delta$ is bounded, $\|L^{-i}\|\leq \lambda^i$ for all $i\geq0$ and $0<\lambda<1$,
one has that the series
$
\bar{x}=x_0+\sum_{i=0}^\infty L^{-i}(r_i)
$
is convergent.
Now, for all $l\geq0$ one has
\begin{eqnarray*}
\|L^l(\bar{x})-x_l\|&=&\left\|L^l(x_0)+\sum_{i=1}^\infty L^{l-i}(r_i)-L^l(x_0)-\sum_{i=0}^{l-1}L^i(r_{l-i})\right\|\\
&=&\left\|\sum_{i=l+1}^\infty L^{l-i}(r_i)+\sum_{i=1}^lL^{l-i}(r_i)-\sum_{i=0}^{l-1}L^i(r_{l-i})\right\|\\
&=&\left\| \sum_{i=l+1}^\infty L^{l-i}(r_i)\right\|\\
&\leq& \sum_{i=l+1}^\infty\lambda^{i-l}\|r_i\|\\
&<& \sum_{i=l+1}^\infty \lambda^{i-l}\delta(L(x_{i-1}))\\
&<& \sum_{i=l+1}^\infty \lambda^{i-l}\delta(x_{i-1})
\end{eqnarray*}
(where the last inequality follows from (iii)).
If $x_l\in B[0,\epsilon(0)]$, then (i) implies
$$
\|L^l(\bar{x})-x_l\|\leq m\sum_{i=l+1}^\infty\lambda^{i-l}=\frac{m}{\lambda^{-1}-1}<\bar{\epsilon}(x_l).
$$
Otherwise, $\|L(x_l)\|\geq \lambda^{-1}\|x_l\|>\bar{\epsilon}(0)$ so $L(x_l)\notin B[0,\bar{\epsilon}(0)]$. Thus,
$\delta(L(x_l))<\frac{\|L(x_l)\|}K$.
It follows that
$\|L(x_l)-x_{l+1}\|< \frac{\|L(x_l)\|}K$ and then
$$
\|x_{l+1}\|\geq \|L(x_l)\|-\|L(x_l)-x_{l+1}\|> \lambda^{-1}\|x_l\|-\frac{\|L\|}K\|x_l\|=P\|x_l\|
$$
where $P=\lambda^{-1}-\frac{\|L\|}K>1$ so $x_{l+1}\notin B[0,\bar{\epsilon}(0)]$.
Repeating the argument we get $\|x_{i+1}\|\geq P\|x_i\|$ and $x_{i+1}\notin B[0,\bar{\epsilon}(0)]$ for all $i\geq l$.
This proves that the sequence $(\|x_i\|)_{i\geq l+1}$ is strictly increasing so
$(\delta(x_i))_{i\geq l+1}$ is strictly decreasing by (i) thus
$$
\|L^l(\bar{x})-x_l\|\leq \delta(x_l)\frac{1}{\lambda^{-1}-1}<(\lambda^{-1}-1)\bar{\epsilon}(x_l)\frac{1}{\lambda^{-1}-1}=\bar{\epsilon}(x_l).
$$
In any case, we have proved that
$$
\|L^l(\bar{x})-x_l\|<\bar{\epsilon}(x_l),\quad\quad\forall l\geq0
$$
and so
$$
\|L^l(\bar{x})-x_l\|<\epsilon(L^l(\bar{x})),
\quad\quad
\forall l\geq0
$$
by Lemma \ref{landru}.
We conclude that $L$ has the positive topological shadowing property.
Since every finite-dimensional Banach space is locally compact, $L$ has the topological shadowing property by Lemma \ref{preparada}.
This proves the result when the spectrum of $L$ lies in $\mathbb{C}\setminus \overline{\mathbb{D}}$.
If the spectrum of $L$ lies in $\mathbb{D}$, then
the one of $L^{-1}$ lies in $\mathbb{C}\setminus \overline{\mathbb{D}}$.
It follows that $L^{-1}$ has the topological shadowing property and so $L$ also does by Proposition 13 in \cite{dlrw}. This finishes the proof.
\end{proof}

\begin{proof}[Proof of Theorem \ref{c-1}]\,
Let $L:X\to X$ be a uniformly expansive linear operator of a Banach space. If $L$ has the topological shadowing property,
then $L$ is hyperbolic by Lemma \ref{putin} and so
the spectrum of $L$ lies either in $\mathbb{D}$ or in $\mathbb{C}\setminus \overline{\mathbb{D}}$ by Lemma \ref{sim}.
\end{proof}

\begin{proof}[Proof of Theorem \ref{c-2}]\,
Let $T:H\to H$ be a normal operator of a Hilbert space.
Then, $T$ is unitary equivalent to a multiplication operator of a measure space
(Theorem 4.6 in \cite{con}). This and Lemma 23 in \cite{lm} imply that $E^c$ is closed.

Now, suppose that $T$ has the topological shadowing property.
Then,
$T$ is expansive by Lemma \ref{baja}.
On the other hand, we say that $x\in H$ is a {\em periodic point} of $T$ if there is $n\in\mathbb{N}$ such that $T^n(x)=x$.
It follows from Lemma 4.6 in \cite{lny} that $\Omega(T)$ is the closure of the periodic points. Clearly every periodic point belongs to $E^c$ and, since $L$ is expansive, $E^c=\{0\}$ so the sole periodic point of $T$ is $0$. Therefore,
$\Omega(T)=\{0\}$ proving the result.
\end{proof}

\noindent

\section*{Acknowledgments}

\noindent
The first author would like to thank the School of Mathematical Sciences of the
Liaoning University, Shenyang, China, for its kindly hospitality during the preparation of this work.

\section*{Funding}

\noindent
YY was partially supported by the National Natural Science Foundation of China (Grant Nos. 12101281) and Scientific Research Foundation of Education Department of Liaoning Province, China (Grant Nos. JYTQN2023190).

\section*{Declaration of competing interest}

\noindent
There is no competing interest.

\section*{Data availability}

\noindent
No data was used for the research described in the article.

\end{document}